\documentclass[reqno,a4paper,11pt]{article}

\usepackage{amsmath,amssymb,amsthm}

\usepackage{mathrsfs}

\usepackage{mathtools}
\usepackage{commath}

\usepackage{thmtools}
\usepackage{thm-restate}

\usepackage{cases}
\usepackage{enumitem}
\setlist[enumerate,1]{label=(\roman*)}

\usepackage[pdftex, pdfborderstyle={/S/U/W 0}]{hyperref}
\hypersetup{
    colorlinks=true,
    linkcolor=magenta,
    citecolor=cyan,
}

\usepackage{cleveref}

\usepackage{etoolbox}

\usepackage{comment}

\usepackage[numbers]{natbib}

\numberwithin{equation}{section}

\ifdefined\thmcolor
\declaretheoremstyle[
  shaded={bgcolor=\thmcolor}
]{plain}
\else
\fi

\ifdefined\defcolor
\declaretheoremstyle[
  headfont=\normalfont\bfseries,
  bodyfont=\normalfont,
  shaded={bgcolor=\defcolor}
]{noital}
\else
\declaretheoremstyle[
  headfont=\normalfont\bfseries,
  bodyfont=\normalfont,
]{noital}
\fi

\declaretheorem[style=plain,numberwithin=section,name=Theorem]{theorem}

\declaretheorem[style=plain,sibling=theorem,name=Lemma]{lemma}
\declaretheorem[style=plain,sibling=theorem,name=Corollary]{corollary}
\declaretheorem[style=plain,sibling=theorem,name=Conjecture]{conjecture}
\declaretheorem[style=plain,sibling=theorem,name=Claim]{claim}

\declaretheorem[style=plain,sibling=theorem,name=Question]{question}
\declaretheorem[style=plain,sibling=theorem,name=Observation]{observation}

\declaretheorem[style=plain,numbered=no,name=Theorem]{theorem-n}
\declaretheorem[style=plain,numbered=no,name=Proposition]{proposition-n}
\declaretheorem[style=plain,numbered=no,name=Lemma]{lemma-n}
\declaretheorem[style=plain,numbered=no,name=Corollary]{corollary-n}
\declaretheorem[style=plain,numbered=no,name=Conjecture]{conjecture-n}
\declaretheorem[style=plain,numbered=no,name=Claim]{claim-n}
\declaretheorem[style=plain,numbered=no,name=Fact]{fact-n}
\declaretheorem[style=plain,numbered=no,name=Open Problem]{openproblem-n}
\declaretheorem[style=plain,numbered=no,name=Question]{question-n}
\declaretheorem[style=plain,numbered=no,name=Observation]{observation-n}

\declaretheorem[style=noital,sibling=theorem,name=Definition]{definition}

\declaretheorem[style=noital,numbered=no,name=Remark]{remark-n}
\declaretheorem[style=noital,numbered=no,name=Definition]{definition-n}
\declaretheorem[style=noital,numbered=no,name=Construction]{construction-n}
\declaretheorem[style=noital,numbered=no,name=Example]{example-n}

\newcommand{\defined}{\mathrel{\coloneqq}}

\DeclarePairedDelimiter{\p}{\lparen}{\rparen}

\newcommand{\st}{\mathbin{\colon}}

\undef{\set}
\DeclarePairedDelimiter{\set}{\lbrace}{\rbrace}

\undef{\emptyset}
\newcommand{\emptyset}{\varnothing}

\DeclarePairedDelimiter{\card}{\lvert}{\rvert}

\newcommand{\union}{\mathbin{\cup}}
\newcommand{\inter}{\mathbin{\cap}}

\DeclareMathOperator{\ind}{\mathbf{1}}

\newcommand{\from}{\colon}

\DeclarePairedDelimiter{\ceil}{\lceil}{\rceil}

\undef{\mod}
\newcommand{\mod}[1]{\ (\mathrm{mod}\ #1)}

\undef{\abs}
\DeclarePairedDelimiterX{\abs}[1]
  {\lvert}{\rvert}{\ifblank{#1}{\,\cdot\,}{#1}}

\undef{\norm}
\DeclarePairedDelimiterX{\norm}[1]
  {\lVert}{\rVert}{\ifblank{#1}{\,\cdot\,}{#1}}

\DeclarePairedDelimiterX{\inner}[2]
  {\langle}{\rangle}{\ifblank{#1}{\,\cdot\,}{#1},\ifblank{#2}{\,\cdot\,}{#2}}

\DeclareMathDelimiter{\given}
  {\mathbin}{symbols}{"6A}{largesymbols}{"0C}

\DeclareMathOperator{\Prob}{\mathbb{P}}
\DeclarePairedDelimiterXPP{\prob}[1]
  {\Prob}{\lparen}{\rparen}{}
  {\renewcommand{\given}{\nonscript\;\delimsize\vert\nonscript\;\mathopen{}}#1}

\DeclareMathOperator{\Expec}{\mathbb{E}}
\DeclarePairedDelimiterXPP{\expec}[1]
  {\Expec}{\lparen}{\rparen}{}
  {\renewcommand{\given}{\nonscript\;\delimsize\vert\nonscript\;\mathopen{}}#1}

\DeclareMathOperator{\Var}{Var}
\DeclarePairedDelimiterXPP{\var}[1]
  {\Var}{\lparen}{\rparen}{}
  {\renewcommand{\given}{\nonscript\;\delimsize\vert\nonscript\;\mathopen{}}#1}

\DeclareMathOperator{\Cov}{Cov}
\DeclarePairedDelimiterXPP{\cov}[2]
  {\Cov}{\lparen}{\rparen}{}{#1,#2}

\newcommand{\sseq}{\subseteq}

\let\l\relax
\newcommand{\l}{\ell}

\newcommand{\FF}{\mathbb{F}}

\newcommand{\NN}{\mathbb{N}}

\newcommand{\cM}{\mathcal{M}}
\newcommand{\cN}{\mathcal{N}}

\newcommand{\cS}{\mathcal{S}}

\usepackage{geometry}
\usepackage{scrextend}
\usepackage[normalem]{ulem}

\usepackage[T1]{fontenc}
\usepackage{titlesec}

\titleformat{\section}{\centering\bfseries\scshape\Large}{\thesection}{1em}{}
\titleformat{\subsection}{\bfseries\scshape\large}{\thesubsection}{1em}{}

\newcommand{\im}{\text{im}}
\newcommand{\setm}[1]{\setminus\set{#1}}

\begin{document}

\title{\textsc{\bfseries Connecting hypercube 1-factors}}

\renewcommand{\thefootnote}{\fnsymbol{footnote}}

\author{\textsc{Lawrence Hollom}\footnotemark[1] \and \textsc{Benedict Randall Shaw}\footnotemark[1]}

\footnotetext[1]{\href{mailto:lh569@cam.ac.uk}{lh569@cam.ac.uk} and \href{mailto:bwr26@cam.ac.uk}{bwr26@cam.ac.uk} respectively. Department of Pure Mathematics and Mathematical Statistics (DPMMS), University of Cambridge, Wilberforce Road, Cambridge, CB3 0WA, United Kingdom}

\renewcommand{\thefootnote}{\arabic{footnote}}

\date{}

\maketitle

\begin{abstract}
    A \emph{1-factorisation} of a regular graph $G$ is a partition of its edge set $E(G)$ into perfect matchings of $G$.
    Behague asked for the minimal \(r=r(d)\) such that some \(1\)-factorisation of the $d$-dimensional hypercube \(Q_d\) has the property that the union of any \(r\) of its 1-factors is connected.
    Previous work by Laufer on perfect \(1\)-factorisations implied that \(r\) is at least three, and Behague gave a construction with \(r=\big\lceil\frac{d}{2}\big\rceil+1\). 
    We improve this upper bound, giving a random construction with \(r=O(\log d)\).
    In other words, we prove the existence of a 1-factorisation \(\cM = \{M_1,\dotsc,M_d\}\) of the hypercube \(Q_d\) such that every \(\cN\subseteq \cM\) of size \(\Omega(\log d)\) is such that \(\bigcup \cN\) is connected.
\end{abstract}


\section{Introduction}
\label{sec:intro}

A \emph{1-factorisation} of a graph \(G\) is a partition \(\cM = \set{M_1,\dotsc,M_k}\) of its edge-set into perfect matchings.
That is, each \(M_i\) is a perfect matching of \(G\) and \(E(G) = \bigcup_{i=1}^k M_i\).
Of particular interest are \emph{perfect} 1-factorisations---those where, for every pair of matchings \(M_i,M_j\in \cM\), their union \(M_i\union M_j\) forms a Hamilton cycle.
Indeed, it is a long-standing conjecture of Kotzig \cite{Kot63} that the complete graph \(K_{2n}\) always has a perfect 1-factorisation. This is known when \(n\) or \(2n-1\) are prime \cite{And73, Kot63}, as well as for finitely many other values of \(n\) (see \cite{Pike19} and the references therein). In particular, cases as small as \(K_{64}\) remain open.

1-factorisations of the hypercube \(Q_d\) have also received significant attention.
However, it is known that a 1-factorisation of \(Q_d\) cannot be perfect, as shown by Laufer \cite{Lau80}.
One may note that a \(1\)-factorisation is perfect if and only if the union of any two of its \(1\)-factors is connected. This led Behague \cite{Beh19} to ask the following question, later repeated in the collection \cite{Imre23}.

\begin{question}[\cite{Beh19}]
\label{q:main}
    For each $d\in\NN$ with $d\geq 3$, let $r = r(d)$ be the minimal positive integer such that there exists a 1-factorization $\cM$ of $Q_d$ where the union of any $r$ distinct 1-factors is connected. 
    What is the value of $r(d)$?
\end{question}

The best known upper bound on \(r(d)\) was \(\ceil{d/2} + 1\), proven by Behague \cite{Beh19}.
Our main result is the following theorem, which significantly improves the upper bound on \(r(d)\).

\begin{theorem}
\label{thm:main}
    The value $r(d)$ is bounded above by $O(\log d)$.
    More precisely, there are constants \(d_0\) and \(c\) such that the following holds.
    For every $d\in\NN$ with \(d\geq d_0\), there is a 1-factorisation $\cM$ of $Q_d$ such that, for any $\cN\sseq \cM$ with $\card{\cN}\geq c \log_2 d$, the union $\bigcup\cN$ is connected.
    In fact, we can take \(d_0 = 3000\) and \(c = 45\).
\end{theorem}

We note in particular that our proof gives a polynomial time randomised algorithm for constructing (with high probability) such a 1-factorisation.

The lack of a perfect 1-factorisation of \(Q_d\) has led to the study of which weaker conditions we can insist that \(\cM\) satisfies. In fact, Laufer also showed a more detailed result---that, if \(H\) is the graph on \(\cM\) formed by connecting \(M_i\) and \(M_j\) iff \(M_i\union M_j\) is a Hamilton cycle, then \(H\) is bipartite. 
A \(1\)-factorisation \(\cM\) is called \emph{semi-perfect} if some matching \(M_1\) has the property that \(M_1\cup M_i\) is a Hamilton cycle for all other \(M_i\)---or equivalently, if \(H\) is the star graph of order \(n\). 
Craft\footnote{Previous papers on this problem \cite{Beh19,Goc10} attribute this conjecture to Craft, citing a webpage of Archdeacon titled ``Problems in Topological Graph Theory,'' of which no surviving copy containing a reference to the conjecture on semi-perfect \(1\)-factorisations could be found.} conjectured that, for \(d\geq 2\), the hypercube \(Q_d\) has a semi-perfect \(1\)-factorisation.
Královič and Královič \cite{Kra05} and Gochev and Gotchev \cite{Goc10} each proved this for odd \(d\), and Chitra and Muthusamy \cite{CM13} proved the full conjecture.

Gochev and Gotchev also introduced the notion of \(k\)-\emph{semi-perfect} 1-factorisations, those for which the graph \(H\) is \(K_{k,d-k}\), and proved these exist for \(k,d\) even. Behague \cite{Beh19} went further, proving that for each \(k<d\), the hypercube \(Q_d\) admits a \(k\)-semi-perfect \(1\)-factorisation. Setting \(k=\big\lfloor\frac{d}{2}\big\rfloor\) gives a \(1\)-factorisation in which any \(\ceil{d/2} + 1\) matchings must contain two whose union is connected, giving the previous upper bound on \(r(d)\).

In \Cref{subsec:outline}, we outline our proof of \Cref{thm:main}. In \Cref{subsec:preliminaries}, we collect some preliminary results on the Chernoff bound and on Hamming codes, both of which we will use in \Cref{sec:proof}, in which we present our proof of \Cref{thm:main}.
Finally, in \Cref{sec:conclusion} we give some concluding remarks and conjecture the value of $r$ for the uniform random \(1\)-factorisation of $Q_d$.


\subsection{Proof outline}
\label{subsec:outline}

Consider the \(1\)-factorisation given by the \emph{directional matchings}: the set of matchings of the form \(M_x\), where \(M_x\) is the set of edges in direction \(x\). We can see that the union of any two directional matchings is a disjoint union of \(4\)-cycles, and likewise that the union of any \(r\) directional matchings is a disjoint union of `small cubes'---copies of \(Q_r\). In particular, for \(r<d\), this is certainly not connected.

Our aim will be to make a large number of random local perturbations to these directional matchings in such a way that the union of any \(r\) matchings is connected with very high probability---indeed, high enough probability to beat a union bound over all \(d\choose r\) possible sets of \(r\) matchings. We begin by choosing a large number of points of \(Q_d\) around which to make these small perturbations---in fact we will take most points of a particular Hamming code \(C\) on \(Q_d\) (reserving some space to later make larger perturbations). This gives a well-behaved, fairly dense set of vertices. Our perturbation at each vertex will simply swap the edges of a random square around that vertex between two matchings. The fact that any two vertices in the code are at distance at least \(3\) will be enough to guarantee that usually these perturbations are compatible---if not, we will ignore the perturbation. We will also make some larger perturbations, swapping within a \(6\)-cube instead, at an exponentially small proportion of vertices---these will be few enough that we can simply choose not to make any other swaps near them.

In fact each `small cube' is quite well connected, as it is a copy of \(Q_r\). Our first step will be to show that it is sufficiently well-connected that, even after our perturbations, each small cube is still connected in the union of our \(r\) perturbed matchings. Indeed, we show that all small cubes are still connected with extremely high probability. We then show that so long as a small cube contains some vertex of the code \(C\), it contains a large number of vertices in the code, and therefore has enough small perturbations that it is connected to all adjacent small cubes.

For many choices of \(r\) matchings, all small cubes will contain such vertices, and thus the entire cube will already be connected. However, the nature of the Hamming code is such that some sets of \(r\) matchings will not yet be connected. Here we show that our small perturbations are enough to connect small cubes into very large components of the graph. These will then be large enough to contain many larger perturbations, which will allow us to bridge the gap between these components and connect the entire cube. Finally, a simple computation will show that we have beaten the union bound, and thus have constructed a \(1\)-factorisation such that the union of any \(r\) matchings is connected.


\subsection{Preliminary results}
\label{subsec:preliminaries}

We will make use of the following form of Chernoff's inequality. (See e.g.\ \cite[Section 28.4] {FK15} for a thorough summary of this style of concentration inequality.) 

\begin{theorem}
    \label{thm:chernoff}
    Let $X$ be a binomial random variable with mean $\mu$. Then
    \begin{equation*}
        \prob{\abs{X - \mu} > t} \leq 2\exp\p[\bigg]{-\frac{t^2}{3\mu}}.
    \end{equation*}
\end{theorem}

We will also make considerable use of Hamming codes in our proofs; we now introduce some relevant notation and then define these codes in the context in which we will use them. 
The definition of the Hamming code requires us to identify the directions of the hypercube with elements of some set \(X\sseq \FF_2^k\setm{0}\) with \(\card{X} = d\). 
As is standard, \(\FF_2\) is the finite field with two elements, \(0\) and \(1\).
As we will often use the details of this identification, we will fix a set \(X\), subject to some further conditions, and index the directions of the hypercube not with \([d]\) but with \(X\).
We call this hypercube \(Q_X\),  with vertex set \(\FF_2^X\) and edges between points differing in only one position.

Both the hypercube \(Q_X\) and the set \(X \sseq \FF_2^k\) have a notion of addition, and these notions are not directly compatible. Since we will use both, we must be careful to specify whether our objects are in \(X\) or \(Q_X\). We therefore write \(b\from X \to Q_X\) for the map that sends \(x\in X\) to the basis element of \(Q_X = \FF_2^X\) in direction \(x\).

In particular, \(X\) and \(Q_X\) are both vector spaces over \(\FF_2\). It will therefore be important to keep explicit which space we are working in. Throughout the paper, we will use the variables \(u,v,w\) to refer to vertices of the hypercube \(Q_X\), and the variables \(x,y,z\) to refer to elements of \(X\).

Given a subset \(X\sseq \FF_2^k\setm{0}\), define the function \(\varphi \from Q_X \to \FF_2^k\) as
\[\varphi(u) \defined \sum_{x\in X} u_x x.\]
The \emph{Hamming code} \(C = C_X\) is then the subset of the \(d\)-dimensional hypercube \(Q_X\) defined as follows.
\[C = C_X \defined \set[\big]{u \in V(Q_X) \st \varphi(u) = 0 \; \text{ in } \; \FF_2^k},\]
where \(u_x\in \FF_2\) is the component of \(u \in \FF_2^X\) in direction \(x\).

We will in practise choose \(k = \ceil{\log_2(d+1)}\) (i.e.\ as large as possible).
Note that any two points of \(C\) are at distance at least 3.
Moreover, note that \(\varphi\) is a group homomorphism (considering \(Q_X = \FF_2^X\) and \(\FF_2^k\) as groups under addition), and thus \(\card{C} = \card{Q_X} / \card{\im(\varphi)}\).
In particular, if \(\varphi\) is surjective then \(\card{C} = 2^{k-d}\).

Note that the set \(X\) is not specified in the above definition.
Indeed, when \(d\) is not one less than a power of 2 we have the freedom to choose \(X\).
We now fix a particular set \(X\) for each \(d\), and refer to the corresponding set \(C_X\) as \emph{the Hamming code} on \(Q_X\).
For reasons which will become clear in \Cref{subsec:large-components}, we will choose \(X\) carefully. Let \(F\sseq \FF_2^k\) be the subset of odd elements of \(\FF_2^k\setm{0}\): those elements with an odd number of 1s.
Noting that \(\card{F} = 2^{k-1} \leq d\), we insist that \(F\sseq X\); the other elements of \(X\) may be chosen arbitrarily.
We call elements of \(\FF_2^k\) which are in \(X\) \emph{active}, and other elements of \(\FF_2^k\) \emph{inactive}.
Note that for this choice of \(X\), the function \(\varphi\) is surjective, and so \(\card{C} = 2^{d-k}\).

We will abuse notation when indexing the set of directions, and, if \(x \in \FF_2^k\setminus X\), then we let \(u_x = 0\). Note that this is consistent with the natural embedding \(Q_X \sseq Q_{\FF_2^k\setm{0}}\).


\section{Proof of main theorem}
\label{sec:proof}


\subsection{Constructing the 1-factorisation}
\label{subsec:factorisation}

Let $C$ be the Hamming code discussed in \Cref{subsec:preliminaries}.

Let $G'\sseq C$ be sampled randomly by taking each element of $C$ independently with probability $2^{-d/10}$.
Then, let $G\sseq G'$ be formed by removing all points of $G'$ which are within distance 14 of another point of $G'$ (i.e.\ if two points are close, then both are removed). 
Note that we expect points of \(G'\) to mostly be pairwise far apart, and so it is likely that only a small proportion of \(G'\) is removed to form \(G\).

Let $H\sseq C$ be formed by removing those points which are within distance 10 of a point of $G'$.
Noting that we expect \(G'\) to cover only an exponentially small proportion of \(C\), we may see that it is likely that almost every vertex of \(C\) is in \(H\).

We construct our 1-factorisation \(\cM\) by starting with \(\cM = (M_x \st x\in X)\), where \(M_y\) is the set of edges in direction \(y\), and then randomly perturbing these matchings.
We note that, even after all of these perturbations, most of the edges in \(M_y\) will still be in direction \(y\).

Assign to each vertex \(u\in C\) two distinct directions \(p_u\) and \(q_u\), and to each \(v\in G\) a set \(\set{r_v^{(1)}, \dotsc, r_v^{(6)}}\) of six distinct directions, all chosen uniformly at random and independently of all other choices.

For each \(u\in H\), we will swap the edges between \(M_{p_u}\) and \(M_{q_u}\) on the square \(\set{u, u + b(p_u), u + b(q_u), u + b(p_u) + b(q_u)}\) as long as this does not interfere with other such swaps.
To be precise, if \(u + b(p_u) + b(q_u)\) is adjacent to an element of \(C\), then let this element be \(w\).
If \(w\) exists and \(w + b(p_w) + b(q_w)\) is adjacent to \(u\) in \(Q_X\), then no swaps are performed around \(u\) (or around \(w\)).
Otherwise, move edges \(\set{u,u + b(p_u)}\) and \(\set{u + b(q_u), u + b(p_u) + b(q_u)}\) from \(M_{p_u}\) to \(M_{q_u}\), and conversely with the other two edges of the square.

Note that here we have actually chosen distinguished directions for all elements of \(C\), not just those in \(H\). This makes very little difference to the construction and its proof, except that it will give us slightly more independence later on for Claim \ref{claim:woolly}.

For \(v\in G\), we will permute edges in the directions \(\set{r_v^{(1)}, \dotsc, r_v^{(6)}}\) around \(v\) on a cube of dimension six.
Indeed, let \(R\) be the 6-cube containing \(v\) with edges in directions \(r_v^{(1)}, \dotsc, r_v^{(6)}\).
For each \(i\in [6]\), move the edges in direction \(r_v^{(i)}\) from \(M_{r_v^{(i)}}\) to \(M_{r_v^{(i+1)}}\) (where indices are interpreted mod 6).
Note that, as elements of \(G\) are pairwise far apart, and also far from all elements of \(H\), these 6-cubes share no edges with any other 6-cubes or squares on which the above operations were carried out.
Thus, once this process is complete, we have a randomised 1-factorisation \(\cM\) of \(Q_X\).

For the rest of the proof of \Cref{thm:main}, we now fix an arbitrary subset \(\cN \sseq \cM\) of our matchings of size \(c k\) for constant \(c\).
We will prove that \(\bigcup \cN\) is connected with probability strictly greater than \(1 - \binom{d}{c k}^{-1}\). 
This allows us to apply a union bound to prove that \(\bigcup \cN\) is, with positive probability, connected for every choice of \(\cN\). 
The above bound will therefore suffice to prove \Cref{thm:main}.

Let \(D\sseq X\) be the set of directions such that \(\cN = \set{M_x \st x\in D}\) and note that \(\card{D} = ck\).
We prove connectivity in three stages: by finding some small components (``small cubes'') which are connected, joining these up into larger components, and then showing that all of these larger components are connected to each other.
First of all, we define these small components.

\begin{definition}
    \label{def:small-cube}
    A \emph{small cube} is a connected subgraph of \(Q_X\) isomorphic to \(Q_{ck}\) with edges in directions in \(D\).
    Let \(\cS_\cN\) be the set of all \(2^{d-ck}\) small cubes, and let \(S_u\in\cS\) be the small cube containing the vertex \(u\in Q_X\).
\end{definition}

Note that the set of small cubes depends on the choice of \(\cN\).


\subsection{Small cubes are connected}
\label{subsec:small-cubes}

In this section we prove that, for any \(\cN \sseq \cM\) of size \(c k\), all small cubes (as defined in \Cref{def:small-cube}) are connected.
We in fact prove a stronger result, without referring to the set \(\cN\).
To cleanly state this result, we call an edge \(e\in E(Q_d)\) \emph{untouched} if it is in direction \(x\) and \(e\in M_x\), i.e.\ \(e\) was not swapped between factors during the construction of the 1-factorisation \(\cM\).

\begin{lemma}
    \label{lem:mostly-untouched}
    Given an edge \(e = \set{u,v}\) of \(Q_X\) with \(u\in C\), and \(d\geq d_0\), the following holds with probability at least \(1 - \exp \p[\big]{-d \log_2 d}\).
    All but at most \(3 \log_2 d\) of the \(d-1\) paths of the form \(u, u + b(x), v + b(x), v\) consist of only untouched edges.
\end{lemma}

\begin{proof}
    First, note that due to the fact that vertices of \(G\) must be pairwise far apart, if any edges in the paths in question are swapped due to proximity to some vertex of \(G\), then all other paths are untouched, at the result follows immediately.
    Thus assume that the only swaps are due to vertices in \(H\).
    
    First we treat the case that one of the paths in question is broken by some system of swaps around a vertex of \(G\). Then \(u,v\) must both be within distance \(7\) of that vertex of \(G\). So certainly no vertices of \(H\) are within distance \(3\) of \(u\) or \(v\), and none of the paths in question are disrupted by swaps due to \(H\).
    
    Indeed, there are no other vertices of \(G\) within distance \(7\) of \(u\) or \(v\). Hence no other vertex of \(G\) is within distance \(6\) of any point on one of the paths in question, so no other vertex of \(G\) disrupts any of the paths in question. But then only one system of swaps affects these paths.
    
    In particular, that \(6\)-cube of swaps can only affect edges in some \(6\) directions \(R=\{r^{(1)},\dots,r^{(6)}\}\). If \(R\) does not contain the direction of \(e\), then these swaps only affect the \(6\) paths for which \(x\in R\). If \(R\) does contain the direction of \(e\), then these swaps may be able to affect the \(5\) paths in which \(x\in R\), as well as any paths such that \(\{u+b(x),v+b(x)\}\) is contained in the \(6\)-cube. If this holds for any path with \(x\notin R\), then the \(6\)-cube is disjoint from all the other paths.
    
    Thus if some system of swaps around a vertex of \(G\) breaks one of the paths in question, then at most \(6\) paths are broken in total. Since we assume \(d\geq d_0\), this is less than \(3 \log_2 d\), as desired. Hence we may assume that the only swaps are due to vertices in \(H\).
    
    As \(u\in C\), we know that at most two of the paths in question can be disturbed due to swaps around \(u\). 
    Recalling that \(C\) is a distance-\(3\) code, the only other swaps which could interfere with these paths are those around vertices in the code adjacent to \(v + b(x)\) for some direction \(x\). 
    Indeed, as each vertex of \(Q_X\) is either in \(C\) or adjacent to at most one vertex from \(C\), there are \(m\leq d\) vertices \(w_1,\dotsc,w_m\in C\) which could lead to swaps interfering with the paths in question.
    However, for a swap around \(w_j\in C\) to interfere with one path, the two directions chosen at \(w_j\) must be among the three directions from \(w_j\) to \(u\).
    Call this event \(A_j\), and note that these events are independent, and each happens with probability \(6/d(d-1)\).

    The probability in question is thus bounded above by
    \[1 - \prob[\Big]{\sum_{j=1}^m \ind[A_j] > 3\log_2 d - 2}.\]
    A simple application of Chernoff's inequality gives us that this is in turn at least
    \[1 - 2 \exp \p[\bigg]{\frac{-(3\log_2 d - 3)^2 d (d-1)}{18 m}}
    \geq 1 - \exp \p{-d \log_2 d},\]
\end{proof}

We may now deduce the following corollary by a union bound.

\begin{corollary}
    \label{cor:small-connected}
    For \(d\geq d_0\), the probability that all small cubes \(S\) are connected by edges of \(\cN\) is at least \(1 - \exp(-d(\log_2 d - 1))\).
\end{corollary}

\begin{proof}
    Firstly, note that there are at most \(2^d\) edges of \(Q_X\) with one end in \(C\), and so by a union bound, the conclusion of \Cref{lem:mostly-untouched} fails for some such edge with probability at most \(\exp \p{-d (\log_2 d - 1)}\).
    Thus assume for the rest of this proof that the conclusion of \Cref{lem:mostly-untouched} holds deterministically.
    
    Consider some small cube \(S\), and some edge \(e=\{u,v\}\in E(S)\) which is not in \(\bigcup \cN\). 
    First suppose \(e\) was removed from \(\bigcup \cN\) by a system of swaps around a vertex of \(G\). 
    Then choose some direction \(x\in D\) which was not one of the six directions affected by that system of swaps. 
    Now the distance conditions on \(G\) and \(H\) imply that the path \(u, u+b(x), v+b(x), v\) is untouched, and so \(u\) and \(v\) are connected.
    
    Hence we may assume \(e\) was removed from \(\bigcup \cN\) by a swap around a vertex \(w\in H\).
    If \(e\) is in direction \(y\), then there is a direction \(z\) to which it was moved, i.e.\ both \(\set{u, u + b(z)}\) and \(\set{v, v+ b(z)}\)are in \(\bigcup \cN\).
    Let \(f\) be the edge \(\set{u + b(z), v + b(z)}\).
    Note that \(w\) must be an endpoint of one of \(e\) and \(f\), and so assume without loss of generality that \(w = u\).
    But this means that the edge \(e\) satisfies the conditions of \Cref{lem:mostly-untouched}, which we have assumed holds deterministically.
    Therefore, as \(ck > 3\log_2 d\), there must be a path of three untouched edges from \(u\) to \(v\), and thus these vertices are connected, as required.
    
    Hence \(S\) is indeed connected, as required.
\end{proof}

We note that the edges used in connecting \(S\) in the above argument need not all be in \(S\).
Indeed, if we connected the ends of \(e\) via an adjacent edge \(f\), then it is possible that \(f\) (and the path connecting the ends of \(f\)) is in an adjacent small cube.
Nevertheless, the small cubes are in any case connected.


\subsection{When adjacent small cubes are connected}
\label{subsec:adjacent}

We will prove that, if a small cube \(S\) has non-empty intersection with \(C\), then in fact \(S\) has a large intersection with \(C\).
Moreover, we will prove that, with high probability, such a small cube \(S\) has large intersection with \(H\), and so there are edges of \(\cN\) connecting \(S\) to all the neighbouring small cubes.
Before proving this lemma, we give some definitions which will be used both in the proof and in later sections.

Recall that \(D\) is the set of directions spanned by small cubes, and so \(D\sseq X \sseq \FF_2^k\setm{0}\).
Thus we may take the subspace \(W\leq \FF_2^k\) to be the linear span of \(D\), and note that \(W\cong \FF_2^\l\) for some \(\l\leq k\).
Let \(E\sseq D\) be a basis of \(W\), and note that \(\card{E} = \l\).

\begin{lemma}
    \label{lem:adjacent-connected}
    Let \(S\in \cS_\cN\) be a small cube such that \(S\inter C \neq \emptyset\).
    Then \(\card{S\inter C} = 2^{ck - \l}\) holds deterministically, and the following both hold with probability at least \(1 - \exp\p[\big]{-2^{ck/3}}\):
    \begin{itemize}
        \item \(\card{S\inter H} \geq 2^{ck/2}\), and
        \item for every small cube \(S'\) adjacent to \(S\), there is an edge of \(\cN\) between \(S\) and \(S'\). 
    \end{itemize}
\end{lemma}

\begin{proof}
    Assume that there is some point \(u\in S\inter C\).
    Unfolding definitions, this is equivalent to \(\varphi(u) = \sum_{x\in X} u_x x = 0\) in \(\FF_2^k\).

    We know that the vertices of \(S\) are sent by \(\varphi\) to some coset \(\varphi[S] = z + W\) of the subspace \(W\sseq \FF_2^k\) of dimension \(\l\) spanned by the directions of \(S\).
    As \(0\in z+W\), we know that \(z + W = W\).
    Thus, to choose a point \(v\in S\inter C\), we may first choose \(v_x\) arbitrarily for directions \(x\) outside the basis \(E\) of \(W\), and we know that there will be a unique choice of \(v_y\) for each \(y\in E\) such that \(v\in C\).
    Thus we see that \(\card{S\inter C} = 2^{ck - \l}\) deterministically, as required.

    For the final two points of the lemma, we will find a large subset of \(S\inter C\) such that the events of these points being in \(H\) and the directions in which edges are swapped around them are totally independent.
    
    \begin{claim}
        \label{claim:woolly}
        Let \(C' \sseq C\) have the property that any two points of \(C'\) are at distance at least \(21\).
        Then the events that \(u \in H\) and that edges in directions \((p_u, q_u) = (x,y)\) are swapped around \(u\) are independent over all choices of \(u \in C'\) and \(x,y\in X\).
    \end{claim}

    \begin{proof}
        These events for some fixed \(u\) depend on membership of \(G'\) for vertices within distance 10 of \(u\), and on \(p_v,q_v\) for points \(v\) within distance 3 of \(u\).

        Membership of \(G'\) and choice of directions \(p_v,q_v\) are independent, and so, if \(C'\) has all pairwise distances at least 21, then the events in question are entirely determined by disjoint sets of independent events, and so are independent.
    \end{proof}

    We now set \(L = 21\) for notational convenience.
    Noting that there are at most \(d^L\) points of \(C\) within distance \(L\) of some given point, we may greedily choose a subset \(C'\sseq C\) with \(\card{C'} \geq 2^{ck-\l} d^{-L}\) with all points of \(C'\) at pairwise distance at least \(L\).
    
    Each point \(u\in C'\) may (independently) fail to be in \(H\) if it is too close to a point of \(G'\).
    This event occurs with probability at most \(1 - d^{10} 2^{-d/10} \geq 1 - O(d^{-2})\).
    Moreover, no edges will be swapped around \(u\) if the directions \(p_u,q_u\) interfere with the directions from another point.
    This can only happen if the directions chosen at \(v\), the point of \(C\) adjacent to \(u + p_u + q_u\) (if it exists and is in \(H\)) are both amongst the three directions between \(u\) and \(v\).
    This has probability \(1 - O(d^{-2})\).

    Thus, noting that \(\card{C'}\geq 2^{(c-1-L)k}\), the probability that a particular set of \(2^{(c-1-L) k - 1}\) of these points fail to be in \(H\) is \(O(d^{-2^{(c-1-L)k}})\).
    Taking a union bound over at most \(2^{2^{(c-1-L)k}}\) choices of this set of points, we see that with probability at least \(1 - O(d^{-2^{(c-1-L)k}})\), we have
    \[\card{S\inter H} \geq 2^{(c-1-L)k - 1} \geq 2^{ck/2},\]
    where we have used that \(c = 45 = 2 L + 3\) in the above bound.

    Let \(S'\) be the small cube adjacent to \(S\) in direction \(x\).
    The directions chosen at the points of \(C'\) are independent, and the probability that a given point swaps a direction within the small cube for \(x\) is at least \(d^{-2}\).
    This only needs to happen at one point of \(C'\) to connect \(S\) and \(S'\).
    The probability that no vertex of \(S\inter H\) connects \(S\) and \(S'\) is thus at most
    \[ (1 - d^{-2})^{\card{S\inter H}} \leq \exp\p[\big]{-d^{-2}2^{ck/2}} \leq \exp\p[\big]{-2^{ck/3}},\]
    as required.
\end{proof}


\subsection{Connected components are large}
\label{subsec:large-components}

Now that we know that every small cube is connected, we work to combine these into larger connected components.
Recall that \(W\leq \FF_2^K\) is the linear span of \(D\), the set of directions spanned by small cubes.
We will show that each coset of \(W\) is connected.

Of course, for many choices of \(D\), the space \(W\) will be the whole of \(\FF_2^k\). In this case, any small cube \(S\) is sent by \(\varphi\) to \(\varphi[S]=W\), which certainly contains \(0\). But then every small cube intersects \(C\), so almost certainly any two adjacent small cubes are connected by an edge of \(\cN\).

We have to work more carefully to treat those \(D\) for which \(W\) is a proper subspace of \(\mathbb{F}_2^k\). In particular, we will have to track when exactly we do have \(0\in \varphi[S]\) in order to connect the small cubes together. Let \(U = W^\perp \cong \FF_2^{k-\l}\) be the orthogonal complement of \(W\) in \(\FF_2^k\). So now viewing \(U\) as \(\FF_2^k/W\), we see that the result of Lemma \ref{lem:adjacent-connected} holds exactly when \(\varphi[S]=0\in U\). We will check this formally later.

In the proof that follows, we will track the behaviour of the special case where \(D\) spans \(\mathbb{F}_2^k\) as a kind of worked example.

For \(s\in U\), define the set \(L_s\sseq\FF_2^k\) as follows.
\[L_s \defined \set[\big]{(x,t)\in W \times U \cong \FF_2^k \st t = s}.\]
So viewing \(U,W\) as subspaces of \(\FF_2^k\), this is just saying that \(L_s\) is the coset \(s+W\). When \(D\) spans the whole of \(\FF_2^k\), this is just the whole of \(\FF_2^k=W\).

We will build our large components out of some intermediate sets, which we define after introducing some more notation.
Let \(Q_{U^*} \defined \FF_2^{U\setm{0}}\) be the hypercube with directions indexed by \(U^* = U\setm{0}\). 
When \(D\) spans the whole of \(\FF_2^k\), these are rather trivial notions: \(U^*\) is the empty set, and so \(Q_{U^*}\) has only one vertex.

In much the same way as we can think of elements of \(Q_X\) as functions from \(X\) to \(\FF_2\), or as vectors indexed by \(X\), we can think of an element \(f\in Q_{U^*}\) as either a function \(f\from U^* \to \FF_2\), or as a vector indexed by \(U^*\).
Similarly to the definition of \(\varphi\), we can define the function \(\psi \from Q_{U^*} \to U \cong \FF_2^{k-\l}\) as
\[\psi(f) \defined \sum_{t \in U} f_t t = \sum_{t \in U^*} f_t t.\]
Our large connected components will be built out of the following sets.
For each \(f\in Q_{U^*}\), define
\[
T_f \defined \set[\Big]{u \in V(Q_X) \st \forall t \in U^*, \; \sum_{z \in L_t} u_z = f_t \; \text{ in } \; \FF_2}.
\]

So for those sets \(D\) that span the whole of \(\FF_2^k\), \(T_f\) is the whole of \(V(Q_X)\).
Note that in general, the sets \(T_f\) partition \(Q_X\) into \(2^{\card{U^*}}\) sets. 
In particular, this is small compared to the size of \(Q_X\), and so we expect these sets \(T_f\) to be large.

We now prove that each \(T_f\) is connected, assuming that the conclusions of \Cref{cor:small-connected} and \Cref{lem:adjacent-connected} hold deterministically.
It will in fact follow from our proof that we can group the sets \(T_f\) together into slightly larger connected components, but we do not pursue this as the sets \(T_f\) are already large enough for the results of \Cref{subsec:connecting} to deal with.

Before proceeding, we make the following observation, which follows from the fact that, if \(u\) varies over the vertices of \(S\), then \(u_z\) is constant for fixed \(z\in L_t\) for \(t \in U^*\).

\begin{observation}
    For a given small cube \(S\in\cS_\cN\), there is a unique \(f\in Q_{U^*}\) such that \(S\sseq T_f\).
\end{observation}

This is to say, the partition \(\cS_\cN\) of the vertices of \(Q_X\) into small cubes is a refinement of the partition into \(T_f\).
We now prove a result concerning when small cubes intersect the Hamming code.

\begin{lemma}
    \label{lemma:f-determines-code-intersection}
    For any small cube \(S\in\cS_\cN\), we have \(\varphi[S] = \psi(f) + W\), and moreover \(S\inter C\) is non-empty if and only if \(\psi(f) = 0\).
\end{lemma}

\begin{proof}
    Note first that \(U,W\leq \FF_2^k\) allows the sum \(\psi(f) + W\) to be interpreted in \(\FF_2^k\).

    Take some \(u\in S\sseq T_f\).
    We have that (under the interpretation \(u_x = 0\) for \(x \in \FF_2^k \setminus X\)),
    \begin{align*}
        \varphi(u) 
        = \sum_{x\in X} u_x x 
        &= \sum_{t \in U} \sum_{x \in L_t} u_x x 
        = \sum_{t \in U} \p[\bigg]{\p[\Big]{\sum_{x\in L_t} u_x \; x|_W} + \p[\Big]{\sum_{x\in L_t} u_x}t} \\
        &= \sum_{t \in U} (y_t + f_t t)
        = \psi(f) + \sum_{t \in U} y_t,
    \end{align*}
    for \(y_t = \sum_{x\in L_t} u_x \; x|_W \in W\), which implies that \(\varphi(u) \in \psi(f) + W\), and so \(\varphi[S]\) is indeed equal to \(\psi(f) + W\).
    Thus \(S\) intersects the Hamming code \(C\) if and only if \(0\in \psi(f) + W\), i.e.\ if and only if \(\psi(f) = 0\), as required.
\end{proof}

We will abuse notation, and write \(\psi(u)\) for \(\psi(f)\), where \(f\in Q_{U^*}\) is the unique value such that \(u \in T_f\).
Recall that, due to \Cref{lem:adjacent-connected}, if \(S\) intersects \(C\) in a point, then the intersection is in fact large.

\begin{lemma}
    \label{lem:big-components}
    Assume that the conclusions of \Cref{cor:small-connected} and \Cref{lem:adjacent-connected} hold deterministically.
    Then \(T_f\) is connected by edges of \(\bigcup \cN\).
\end{lemma}

\begin{proof}
    Given two points \(u,v\in T_f\), we prove that there is a path between them using edges of \(\bigcup \cN\), assuming only that all small cubes are connected, and that two adjacent small cubes are connected provided that one of them has non-empty intersection with the Hamming code \(C\).

    As small cubes are connected, we may assume that \(u\) and \(v\) agree on directions in \(D\), so it suffices to find a path covering the other directions.
    We will construct a sequence of points \(u = u_0,u_1,\dotsc,u_{m-1},u_m = v\), where there is a path from each \(u_i\) to \(u_{i+1}\), and for all \(1\leq i \leq m-1\) we have \(\psi(u_i) = 0\).

    If \(\psi(f) = 0\), then by \Cref{lemma:f-determines-code-intersection} every small cube in \(T_f\) intersects \(C\), and thus by assumption is connected to every adjacent small cube; in this case we will take \(u_1 = u_0 = u\) and \(u_{m-1} = u_m = v\).
    
    If \(\psi(f) \neq 0\), then we first claim that \(S_u\) is adjacent to a small cube \(S_{u_1}\), where \(\psi(u_1) = 0\).
    Indeed, we know in this case that \(\psi(f)\in U^*\). Due to the definition of the Hamming code \(C\), and in particular the fact that the set \(F\) of odd elements of \(\FF_2^k\) has \(F\sseq X\), there are active elements of \(\FF_2^k\) in \(L_s\) for every \(s\in U\).
    Thus, in particular, we can fix some active \(x\in L_{\psi(f)}\inter X\), and let \(u_1 = u_0 + b(x)\) (recalling that \(b\) is the map that sends \(x\in X\) to the basis element of \(Q_X = \FF_2^X\) in direction \(x\)), so \(\psi(u_1) = \psi(u_0) + x|_U = 0\), as required.

    Now, given \(u_i\), we construct \(u_{i+1}\), which will differ from \(u_i\) in exactly one or two directions.
    Take some direction \(y\in X\) with \(y\neq x\) in which \(u_i\) and \(v\) differ, noting that if this direction does not exist then we are done.
    
    If \(y \in W\), then let \(u_{i+1} = u_i + b(y)\) and note that \(\psi(u_{i+1}) = \psi(u_i) + y|_U = 0\).
    The small cubes \(S_{u_i}\) and \(S_{u_{i+1}}\) are thus connected to each other, and so there is a path in \(\bigcup \cN\) from \(u_i\) to \(u_{i+1}\), as required.
    
    Otherwise, \(y \notin W\), and note that, due to parity considerations, there must be some other direction \(z\neq x\) with \(z + W = y + W\) in which \(u_i\) and \(v\) differ.
    Let \(u_{i+1} = u_i + b(y) + b(z)\), and note that 
    \[\psi(u_{i+1}) = \psi(u_i) + y|_U + z|_U = \psi(u_i) = 0.\]
    Thus both \(S_{u_i}\) and \(S_{u_{i+1}}\) are connected to all neighbouring small cubes; in particular, they are both connected to \(S_{u_i+b(y)}\).
    Therefore there are paths from \(u_i\) and \(u_{i+1}\) to \(u_i + b(y)\), and thus there is a path from \(u_i\) to \(u_{i+1}\), as required.

    Finally, we can use direction \(x\) again if necessary to find a path from \(u_{m-1}\) to \(v\).
    Connecting these paths together in sequence produces a path connecting \(u\) and \(v\), and proving that \(T_f\) is connected, as required.
\end{proof}

Finally, we note that the sets \(T_f\) are indeed large.

\begin{lemma}
    \label{lem:big-components-are-big}
    For all \(f\in Q_{U^*}\), we have
    \[\card{T_f} \geq 2^{d(1 - 2/\log_2 d)}.\]
\end{lemma}

\begin{proof}
    By symmetry, we know that
    \begin{align*}
        \card{T_f} 
        = \frac{2^d}{\card{Q_{U^*}}} 
        = 2^{d - 2^{k-\l} + 1}
        \geq 2^{d - 2(d+1)/\log_2 d + 1}
        \geq 2^{d(1 - 2/\log_2 d)},
    \end{align*}
    as required.
\end{proof}


\subsection{Connecting the large components}
\label{subsec:connecting}

We prove the following lemma.

\begin{lemma}
    \label{lem:large-connections}
    If \(f\in Q_{U^*}\) is such that \(T_f\inter C \neq \emptyset\), then with probability at least \(1 - \exp\p[\big]{-2^{d/2}}\), for every \(g\in Q_{U^*}\) within a distance 3 of \(f\), \(T_f\) is connected to \(T_g\).
\end{lemma}

\begin{proof}
    First note that if \(T_f\inter C\) is nonempty, then every small cube \(S\sseq T_f\) has \(\card{S\inter C} = 2^{ck - \l} \geq 2^{(c-1)k}\).
    This implies that
    \[\card{T_f\inter C} \geq 2^{-k} \card{T_f}  \geq 2^{d(1 - 2/\log_2 d) - k}.\]
    In particular, following similar lines to the proof of \Cref{lem:adjacent-connected}, for any constant \(L\) there is a subset \(C'\sseq T_f\inter C\) with points at pairwise distance at least \(L\) and 
    \[\card{C'} \geq 2^{d(1 - 2/\log_2 d) - k - L\log_2 d} > 2^{3d/4}.\]
    For sufficiently large \(L\), the events of points of \(C'\) being in \(G\) and the directions in which they swap edges are all independent.
    In particular, \(L=15\) suffices, and noting that \(k < 2\log_2 d\), we find that the above inequality holds for all \(d \geq 2\).

    For \(v\in C'\), let \(B_v\) be the event that \(v\in G\) and the swapping of edges around \(v\) results in a path from \(T_f\) to \(T_g\) (noting that every element of \(T_f\) is within distance 3 of some point of \(T_g\)).
    Recall that a point \(v\in C\) is in \(G\) if it is selected to be in \(G'\) (which happens with probability \(2^{-d/20}\)), and no point other within distance 10 of \(v\) is selected to be in \(G'\).
    Given that \(v\in G\), we see that the connection of \(T_f\) and \(T_g\) occurs with probability at least \(d^{-6}\).
    The event \(B_v\) thus occurs with probability
    \[\prob{B_v} \geq 2^{-d/20} (1 - 2^{-d/20})^{d^{10}} d^{-6}
    \geq 2^{-d/20 - 2^{1-d/20} d^{10} -6 \log_2 d}
    > 2^{-d/10},\]
    where we have assumed that \(d\) is sufficiently large.
    In particular, \(2^{1-d/20} d^9 < 1/40\) for \(d \geq 2115\), and \(6 \log_2 d < d/40\) for \(d \geq 2741\), so \(d \geq 3000\) suffices.
    Moreover, these events are independent, and so the probability that none of them occur is at most
    \[(1 - 2^{-d/10})^{2^{3d/4}} < \exp(-2^{-d/10})^{2^{3d/4}} < \exp\p[\big]{-2^{d/2}},\]
    from which the required result follows immediately.
\end{proof}

\begin{corollary}
    \label{cor:all-large-connections}
    With probability at least \(1 - \exp\p[\big]{-2^{d/4}}\), for any \(f,g\in Q_{U^*}\) there is a path from \(T_f\) to \(T_g\), under the assumption that each \(T_h\) is connected.
\end{corollary}

\begin{proof}
    By a simple union bound, we have that the conclusion of \Cref{lem:large-connections} holds for every \(f\) with probability at least \(1 - 2^d \exp\p{-2^{d/2}} > 1 - \exp\p{-2^{d/4}}\), so we now assume that this holds deterministically.

    Fix \(f,g\in Q_{U^*}\).
    We can find a path from \(f\) to \(g\) in \(Q_{U^*}\) by first stepping to some \(h_1\) such that \(T_{h_1}\inter C\neq \emptyset\), noting that, as in the proof of \Cref{lem:big-components}, each small cube either intersects \(C\) or is adjacent to a small cube which intersects \(C\). 
    Then by assumption, \(T_f\) is connected to \(T_{h_1}\).
    Then we may find a sequence \(h_1,h_2,\dotsc,h_t\) such that for each \(i\) the distance from \(h_i\) to \(h_{i+1}\) is exactly 3, \(T_{h_i} \inter C\) is non-empty, and either \(h_t = g\) or \(h_t\) is adjacent to \(g\).
    Then our assumption tells us that \(T_{h_i}\) is connected to \(T_{h_{i+1}}\) for every \(i\), and \(T_{h_t}\) is connected to \(T_g\).
\end{proof}


\subsection{Deducing the result}
\label{subsec:proof-of-main}

Finally, we may put the above ingredients together finish the proof.

\begin{proof}[Proof of \Cref{thm:main}]
    By a union bound, the conclusions of \Cref{cor:small-connected}, \Cref{lem:adjacent-connected}, and \Cref{cor:all-large-connections} hold simultaneously with probability at least
    \[1 - \exp(-d(\log_2 d - 1)) - \exp\p[\big]{-2^{ck/3}} - \exp\p[\big]{-2^{d/4}}.\]
    
    Indeed, \(\exp(-d (\log_2 d - 1))\) and \(\exp(-2^{d/4})\) are both bounded above by \(2^{-d-2}\) for \(d \geq 14\).
    We also have \(\exp(-2^{ck/3}) < 2^{-d-2}\) if and only if \(2^{ck/3} > (d+2) \ln 2\), and as \(k > \log_2 d\) and \(c > 6\), we have \(2^{ck/3} > d^2 > (d + 2) \ln 2\) for \(d\geq 3\).
    Therefore these three results all hold with probability at least \(1-2^{-d} > 1 - \binom{d}{ck}^{-1}\), the required value for our union bound over all choices of \(\cN\).
    
    In the case that all three of these results hold, we know by \Cref{lem:big-components} that every set \(T_f\) is connected, and by \Cref{cor:all-large-connections} that any \(T_f\) and \(T_g\) are connected to each other by edges of \(\cN\).
    Thus the whole of \(Q_d\) is connected by edges of \(\cN\), as required.
\end{proof}


\section{Concluding remarks}
\label{sec:conclusion}

We have found a randomised 1-factorisation of the hypercube \(Q_d\) for \(d\geq 3000\) such that, with high probability, any subset of \(45\log_2 d\) of these matchings connects the whole hypercube.
This leaves the state of the art (for large \(d\)) at
\[ 3 \leq r(d) \leq 45\log_2 d. \]

There is still a significant gap between these two bounds, and it is not at all clear what the true growth rate of \(r(d)\) should be---it is entirely plausible that there is in fact a constant upper bound.
Pushing the upper bound below \(\log_2 d\) seems to be beyond the reach of random constructions such as that considered here, at least without some significant new idea.
However, based on the intuition that the limiting factor in such random constructions should be relatively small-scale obstructions to connectivity, we make the following conjecture.

\begin{conjecture}
    Let \(\cM\) be a 1-factorisation of \(Q_d\) sampled uniformly at random from the set of all such 1-factorisations, and let \(r(\cM)\) be the minimal integer \(r\geq 1\) such that the union of any \(r\) distinct 1-factors of \(\cM\) is connected.
    Then, with high probability, \(r(\cM) = (1 + o(1)) \log_2 d\).
\end{conjecture}

Any improvement of the lower bound, even to a larger constant, would also be of great interest.


\section{Acknowledgements}
\label{sec:acknowledgements}

Both authors are funded by the Internal Graduate Studentship of Trinity College, Cambridge.


\bibliographystyle{abbrvnat}  
\renewcommand{\bibname}{Bibliography}
\bibliography{main}




\end{document}